\documentclass[12pt,reqno,a4paper]{amsart}
\usepackage{amsmath}
\usepackage{amssymb}
\usepackage{stmaryrd}
\usepackage{amsthm}

\usepackage[applemac]{inputenc}

\advance\voffset    by -1  cm
\advance\hoffset    by -1.5cm
\advance\textwidth  by  3  cm
\advance\textheight by  1  cm
\sffamily

\usepackage{pstricks}
\usepackage{multido}

\usepackage{graphicx}
\usepackage{color}

\newcommand{\showgrid}{}

\newcommand{\gridon}{\renewcommand{\showgrid}{\psset{subgriddiv=1,griddots=10,gridlabels=6pt}\psgrid}}

\gridon 

\newgray{hellgrau}{0.9}
\newrgbcolor{dummycolor}{0.75 0.85 0.0}

\long\def\comment#1{%
\psshadowbox[%
fillstyle=solid,fillcolor=dummycolor,linewidth=0.05%
]{\hbox to 12.70cm {\vbox{\strut {\bf Kommentar:} #1\hfill}}}%
}

\def\figref#1{Figure~\ref{#1}}
\def\heute{\number\day.~\ifcase\month\or
 J\"anner\or Februar\or M\"arz\or April\or Mai\or Juni\or
 Juli\or August\or September\or Oktober\or November\or Dezember\fi
 \space\number\year}

\def\defeq{:=}

\def\setof#1{\left\{#1\right\}}
\def\of#1{\!\left(#1\right)}

\def\pas#1{\left(#1\right)}

\def\bit{\begin{itemize}}
\def\eit{\end{itemize}}
\def\beq{\begin{equation}}
\def\eeq{\end{equation}}

\def\1{{\mathbf 1}}

{\begin{list}{#1}{\parsep0em \itemsep0.3em \labelwidth1em
\labelsep0.5em \leftmargin1.5em }} {\end{list}}

\def\path{{%\mathrm 
P}}
\def\figref#1{Figure~\ref{#1}}

\def\EM#1{{\em #1\/}}

\newif\ifenglish
\englishtrue

\newtheorem{ex}{\ifenglish Example\else Beispiel\fi}

\newtheorem{thm}{\ifenglish Theorem\else Satz\fi}

\newtheorem{lem}{Lemma}

\newtheorem{cor}{\ifenglish Corollary\else Korollar\fi}

\theoremstyle{remark}
\newtheorem{obs}{\ifenglish Observation\else Beobachtung\fi}

\def\weight{\omega}

\def\and{\wedge}

\def\myrange#1{\setof{1,\dots,k}}

\def\schurf{s}
\def\length#1{\ell\of{#1}}
\def\tableau{T}
\def\indexthis#1#2{{#1}^{(#2)}}
\def\sapadd#1#2#3{{#1}\!+^{#2}_{#3}}
\def\sappeel#1{\left.#1\right\downarrow}
\def\sappeeldown#1#2{\left.#1\right\downarrow^{#2}}
\def\sappeelup#1#2{\left.#1\right\uparrow_{(#2)}}
\def\cop{{\sl\bf cp}}
\def\coc{{\sl\bf cocp}}
\def\pointconf{configuration of (starting/ending) points (short: \cop){\gdef\pointconf{\cop}}}
\def\pointarr{circular orientation of coloured (starting/ending) points (short: \coc){\gdef\pointarr{\coc}}}

\newcount\labelcount

\newif\iflongversion
\longversionfalse

\begin{document}

\bibliographystyle{plain}

\title{Bijective proofs for Schur function identities}

\begin{abstract}
In \cite{gurevichPyatovSaponov:2009}, Gurevich, Pyatov and Saponov stated an
expansion for the product of two Schur functions and gave a proof based on the
Pl\"ucker relations.

Here we show that this identity is in fact a special case of a quite general
Schur function identity, which was stated and proved 
in \cite[Lemma~16]{fulmek:2001}. In \cite{fulmek:2001}, it
was used to prove bijectively Dodgson's condensation formula and the Pl\"ucker relations,
but was not paid further attention: So we take this opportunity to
make obvious the range of applicability of this identity by giving concrete examples, accompanied by many graphical
illustrations.
\end{abstract}

\author{Markus Fulmek}
\address{Fakult\"at f\"ur Mathematik, 
Nordbergstra\ss e 15, A-1090 Wien, Austria}
\email{{\tt Markus.Fulmek@Univie.Ac.At}\newline\leavevmode\indent
{\it WWW}: {\tt http://www.mat.univie.ac.at/\~{}mfulmek}
}

\date{\today}
\thanks{
Research supported by the National Research Network ``Analytic
Combinatorics and Probabilistic Number Theory'', funded by the
Austrian Science Foundation. 
}

\maketitle

\section{Introduction}
\label{sec:intro}

In \cite{gurevichPyatovSaponov:2009}, Gurevich, Pyatov and Saponov stated an
expansion for the product of two Schur functions and gave a proof based on the
Pl\"ucker relations. Here we show that this identity is in fact a special case
of a more general Schur function identity \cite[Lemma 16]{fulmek:2001}. Since
this involves a process of ``translation'' between the languages of \cite{fulmek:2001}
and \cite{gurevichPyatovSaponov:2009} which might not be self--evident, we
explain again the corresponding combinatorial constructions here. These constructions
are best conceived by pictures, so we give a lot of figures illustrating the concepts.

This paper is organized as follows: 

In Section~\ref{sec:background} we recall the basic definitions (partitions,
Young tableaux, skew Schur functions and nonintersecting lattice paths).

In Section~\ref{sec:bicoloured}, we present the central bijective construction
(recolouring of bicoloured paths in the overlays of families
of nonintersecting lattice paths corresponding to some product of skew
Schur functions) and show how this yields a quite general Schur function
identity (Theorem~\ref{lem:fulmek2}, a reformulation of \cite[Lemma 16]{fulmek:2001}).

In Section~\ref{sec:applications} we try to exhibit the broad range of
applications of Theorem~\ref{lem:fulmek2}: In particular, we show how the
identity \cite[(3.3)]{gurevichPyatovSaponov:2009} appears as (a translation of) a
special case of Theorem~\ref{lem:fulmek2}.

\section{Basic definitions}
\label{sec:background}

An infinite weakly decreasing series of nonnegative integers $\pas{\lambda_i}_{i=1}^\infty$,
where only finitely many elements are positive, is called a \EM{partition}. The largest
index $i$ for which $\lambda_i>0$ is called the \EM{length} of the partition $\lambda$ and
is denoted by $\length{\lambda}$. For convenience, we shall in most cases omit the trailing
zeroes, i.e.,
for $\length{\lambda}=r$ we simply write $\lambda=\pas{\lambda_1,\lambda_2,\dots,\lambda_r}$,
where $\lambda_1\geq\lambda_2\geq\dots\geq\lambda_r> 0$.

The {\em Ferrers diagram\/} $F_\lambda$
of $\lambda$
is an array of cells with $\length{\lambda}$ left-justified rows and $\lambda_i$
cells in row $i$.

An {\em $N$--semistandard Young tableau\/} of shape $\lambda$ is a filling of the cells
of $F_\lambda$ with integers from the set $\{1,2,\dots,N\}$, such
that the numbers filled into the cells weakly increase in rows
and strictly increase in columns.

Let ${\tableau}$ be a semistandard Young tableau and define $m({\tableau}, k)$ to be
the number of entries $k$ in ${\tableau}$. Then the weight $w({\tableau})$ of ${\tableau}$
is defined as follows:
\begin{equation*}
\weight\of{\tableau} = \prod_{k = 1}^{N} x_k^{m({\tableau}, k)}.
\end{equation*}

\EM{Schur functions}, which are irreducible general linear characters, can be combinatorially defined by means of $N$--semistandard Young tableaux (see, for instance,
\cite[Definition~4.4.1]{sagan:2000}):

\begin{equation*}
\schurf_\lambda(x_1, x_2, x_3, \dots, x_N) = \sum_{{\tableau}}\weight\of{\tableau},
\end{equation*}
where the sum is over all $N$--semistandard Young tableaux ${\tableau}$ of
shape $\lambda$. 

Consider some partition $\lambda=(\lambda_1,\dots,\lambda_r>0)$, and let 
$\mu$ be a partition such that  $\mu_i\leq\lambda_i$ for all $i\geq 1$.
The {\em skew Ferrers
diagram\/} $F_{\lambda/\mu}$ of $\lambda/\mu$
is an array of cells with $r$ left-justified rows and $\lambda_i-\mu_i$
cells in row $i$, where the first $\mu_i$ cells in row $i$ are missing.
An {\em $N$--semistandard skew Young tableau\/} of shape $\lambda/\mu$ is a filling of the cells of $F_{\lambda/\mu}$ with integers from the set
$\{1,2,\dots,N\}$,
such that the numbers filled into the cells weakly increase in rows and
strictly increase in columns
(see the left picture of Figure~\ref{fig:skew-Young-tableau} for an illustration).

\begin{figure}
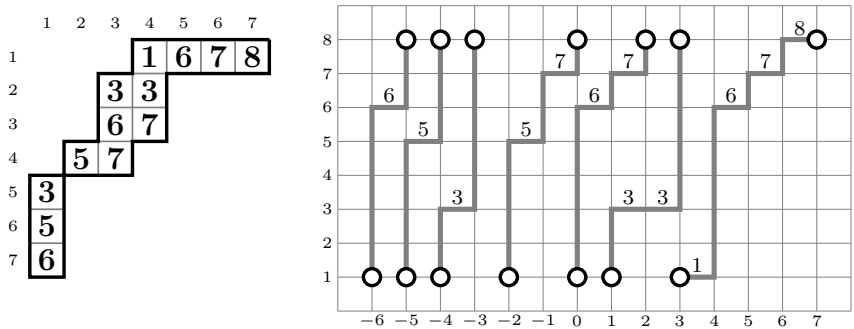

\caption{The left picture presents a semistandard Young tableau $\tableau$ of skew shape $\lambda/\mu$, where
$\lambda=\pas{7, 4, 4, 3, 1, 1, 1}$ and $\mu=\pas{3, 2, 2, 1}$. Assuming that the
entries of $\tableau$ 
are chosen from $\setof{1,2,\dots,8}$ (i.e.: $\tableau$ is an $8$--semistandard Young tableau), the right picture
shows the corresponding family of $7=\length{\lambda}$ nonintersecting lattice paths (with shift $t=0$): Note that the height
of the $j$--th horizontal step in the $i$--th path (the paths are counted from right to left) is equal
to the $j$--th entry in row $i$ of $\tableau$.}
\label{fig:skew-Young-tableau}
\input graphics/ferrers
\end{figure}

Then we can define the \EM{skew Schur function}:

\begin{equation}
\label{eq:skewSchur}
\schurf_{\lambda/\mu}(x_1, x_2, x_3, \dots, x_N) = \sum_{{\tableau}}\weight\of{\tableau},
\end{equation}
where the sum is over all $N$--semistandard skew Young tableaux ${\tableau}$ of
shape $\lambda/\mu$, where the weight $\weight\of{\tableau}$ of ${\tableau}$
is defined as before.

Note that for $\mu=\pas{0,0,\dots}$ the skew Schur function $\schurf_{\lambda/\mu}$ is
identical to the ``ordinary'' Schur function $\schurf_\lambda$.

The Gessel-Viennot interpretation \cite{gessel-viennot:1998} gives an equivalent
description of a semistandard Young tableau $\tableau$ of shape $\lambda/\mu$ as an
$r$--tuple $P=\pas{p_1,\dots,p_{r}}$ of nonintersecting
lattice paths, where $r\defeq\length{\lambda}$: Fix some (arbitrary) integer \EM{shift}
$t$ and consider paths in the lattice $\mathbb Z^2$ (i.e., in the directed
graph with vertices $\mathbb Z\times\mathbb Z$ and arcs from $(j, k)$ to $(j+1, k)$ and
from $(j, k)$ to $(j, k+1)$ for all $j, k$). The $i$--th path $p_i$ starts at $(\mu_i-i+t, 1)$
and ends at $(\lambda_{i}-i+t, N)$, and the $j$--th \EM{horizontal} step in $p_i$ goes from
$\pas{-i+t+j-1,h}$ to $\pas{-i+t+j,h}$, where $h$ is the $j$--th entry in row $i$ of
$\tableau$. Note that the conditions on the entries of $\tableau$ imply that
no two paths $p_i$ and $p_j$ thus defined have a lattice point in common: such an
$r$-tuple of paths is called \EM{nonintersecting}
(see the right picture of Figure~\ref{fig:skew-Young-tableau} for an illustration).

In fact, this translation of tableaux to nonintersecting lattice paths is
a \EM{bijection} between the set of all $N$--semistandard Young tableaux of shape $\lambda/\mu$
and the set of all $r$--tuples of nonintersecting lattice paths with starting and
ending points as defined above. This bijection is \EM{weight preserving} if we define
the weight of an $r$--tuple $P$of nonintersecting lattice paths in the obvious way, i.e., as
\begin{equation*}
\weight\of P \defeq\prod_{k = 1}^{N} x_k^{n(P,k)},
\end{equation*}
where $n(P,k)$ is the number of horizontal steps at height $k$ in $P$. So in the definition
\eqref{eq:skewSchur} we could equivalently replace symbol ``$\tableau$'' by symbol ``$P$'',
and sum over $r$--tuples of lattice
paths with prescribed starting and ending points instead of tableaux with prescribed shape.

Note that the horizontal coordinates of starting and ending points determine uniquely the
\EM{shape} $\lambda/\mu$ of the tableau, and the vertical coordinate (we shall call the 
vertical coordinate of points the \EM{level} in the following) of the ending points determines
uniquely the \EM{set of entries} $\setof{1,2,\dots,N}$ of the tableau. (The choice of the
shift parameter $t$ does influence neither the shape nor the set of entries.) 

\section{Bicoloured paths and products of skew Schur functions}
\label{sec:bicoloured}

In the following, all skew Schur functions are considered as functions of the variables
$\pas{x_1,\dots,x_N}$. (Equivalently, all tableaux have entries from the set
$\setof{1,\dots,N}$, and all families of nonintersecting lattice paths have ending
points on level $N$).

Viewing the product of two skew Schur functions
$$
\schurf_{\lambda/\mu}\cdot\schurf_{\sigma/\tau}
$$
as the generating function of ``overlays of two families of nonintersecting lattice paths''
(according to definition~\eqref{eq:skewSchur}) gives rise to a bijective construction,
which (to the best of our knowledge) was first used by Goulden \cite{goulden:1988}.
This construction was used
in \cite{fulmek:2001} to describe and prove a class of Schur function identities, special
cases of which imply Dodgson's condensation formula and the Pl\"ucker relations.

\begin{figure}
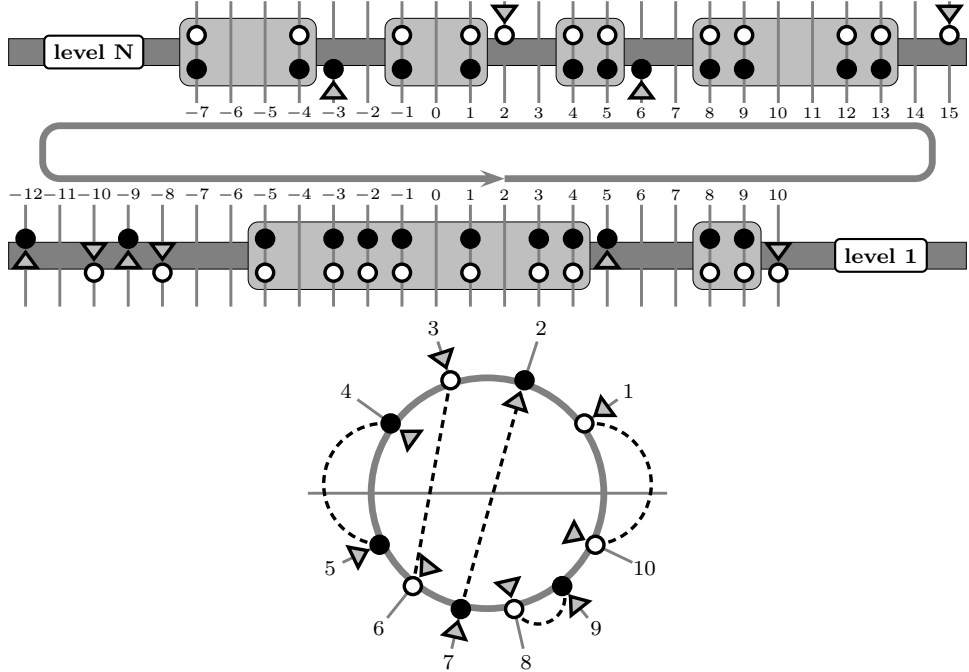

\caption{Illustration for the example given in Section~\ref{sec:bicoloured}.}
\label{fig:twoSkewShapes}
\input graphics/twoSkewShapes
\end{figure}

We shall present this construction by way of an example: Consider skew shapes
$\lambda/\mu$, where
\begin{align}
\lambda &= \pas{14, 13, 13, 11, 11, 9, 9, 8, 8, 7, 5, 3},\notag\\
\mu &= \pas{9, 9, 9, 6, 6, 5, 4, 4, 4, 3, 1},\label{eq:lambdamu}
\end{align}
and $\sigma/\tau$, where
\begin{align}
\sigma &= \pas{14, 14, 12, 12, 11, 11, 11, 9, 8, 7, 7, 5},\notag\\
\tau &= \pas{10, 10, 8, 8, 8, 7, 6, 6, 6, 5, 2}.\label{eq:sigmatau}
\end{align}
For the skew shape $\lambda/\mu$, choose fixed shift $2$, and for the skew shape
$\sigma/\tau$ choose fixed shift $0$, and consider the starting and ending points
of the corresponding families of nonintersecting lattice paths. For instance,
the ending point of the first path corresponding to $\lambda/\mu$ is $\pas{\lambda_1-1+2,N}=\pas{15,N}$, and the starting point of the last (twelfth) path corresponding to $\lambda/\mu$
is $\pas{\mu_{12}-12+2,1}=\pas{-10,1}$, and the ending point of the first path corresponding
to $\sigma/\tau$ is $\pas{\sigma_1-1,N}=\pas{13,N}$.

Now colour the starting/ending points corresponding to $\lambda/\mu$ \EM{white}, and
the starting/ending points corresponding to $\sigma/\tau$ \EM{black}:
See the upper picture of \figref{fig:twoSkewShapes},
where the starting/ending points of $\lambda/\mu$ are drawn as white circles,
and the starting/ending points of $\sigma/\tau$ are drawn as black circles.

All starting/ending points which are coloured both black \EM{and} white are never affected
by the following constructions: In the upper picture of \figref{fig:twoSkewShapes}, these points are enclosed by grey rectangles.

We call the remaining starting/ending points (which are coloured \EM{either} black \EM{or}
white) the \EM{coloured} points. Note that the number of coloured points is necessarily
\EM{even}.

For the coloured points, assume the circular orientation ``from right to left along level $N$,
and then from left to right along level $1$''.
In the upper picture of \figref{fig:twoSkewShapes}, this circular orientation is indicated by a grey circular
arrow.

Furthermore, assign to paths corresponding to
$\lambda/\mu$ the orientation \EM{downwards}, and to paths corresponding to
$\sigma/\tau$ the orientation \EM{upwards}. In the upper picture of \figref{fig:twoSkewShapes}, this orientation of paths is indicated by upwards or downwards pointing triangles.

If we focus on the coloured points, we may encode the situation in a simpler
picture, where the coloured starting/ending points are located on the lower/upper half of a \EM{circle}, and where the orientation of
the respective path is translated to a \EM{radial orientation} (either towards the
center of the circle or away from it). The lower picture of \figref{fig:twoSkewShapes}
illustrates this: A grey horizontal line indicates the separation of the lower and upper half
of the circle, the point labeled $1$ corresponds to the lattice point $\pas{15,N}$,
the point labeled $2$ corresponds to the lattice point $\pas{6,N}$, and so on.

Now consider some pair $\pas{P_1,P_2}$ of families of nonintersecting lattice paths,
where $P_1$ corresponds to some tableau of shape $\lambda/\mu$, and $P_2$ corresponds to
some tableau of shape  $\sigma/\tau$.
We call the paths of $P_1$ the \EM{white} paths and the paths of $P_2$ the \EM{black} paths,
and we colour the arcs of the lattice $\mathbb Z^2$ accordingly (i.e., arcs used by some
white path are coloured white, and arcs used by some black path are coloured black).
As with the starting/ending points, arcs which are coloured black \EM{and} white are
not affected by the following construction, and we call all arcs which are
\EM{either} black \EM{or} white the \EM{coloured} arcs. We 
construct \EM{bicoloured paths}
\bit
\item connecting (only) coloured starting/ending points
\item and using (only) coloured arcs
\eit
by the following algorithm: 
\begin{quote}
We start at some coloured point $q$ and follow
the path determined by the \EM{unique} coloured arc incident with it in the respective orientation (i.e., either up/right or down/left). Whenever we meet \EM{another} path on our way (necessarily, this path is of the other colour),
we ``change colour and orientation'', i.e., we follow this new path \EM{and} change the orientation (i.e., if we were moving up/right along the old path, we move down/left
along the new path, and vice versa). We stop if there is no possibility to go further.
\end{quote}

\begin{figure}
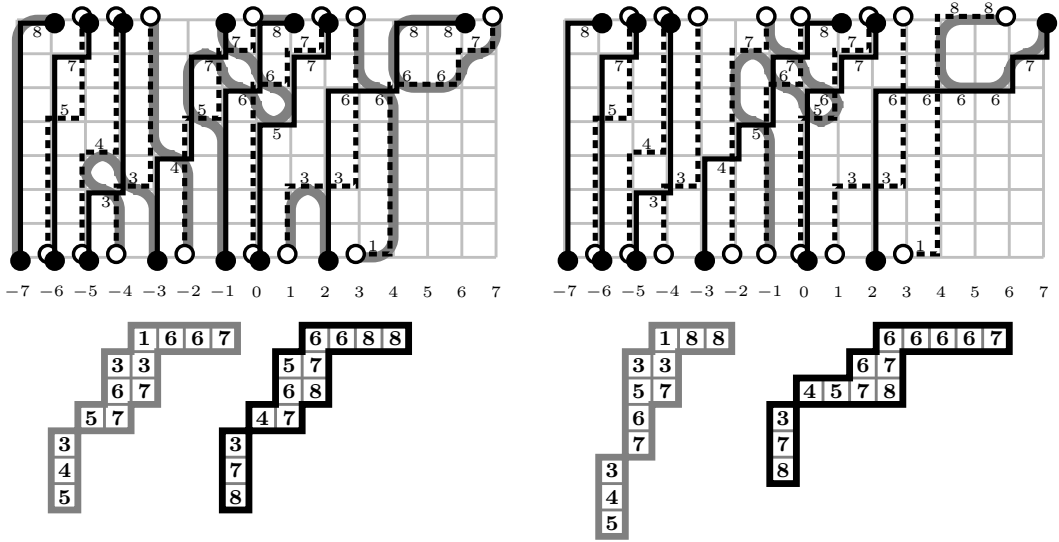

\caption{
The left picture presents two tableaux of the identical shape
$\pas{7, 4, 4, 3, 1, 1, 1}/\pas{3, 2, 2, 1}$ in the lower part, the upper part shows the
corresponding overlay of black and white paths (white paths are indicated by dashed lines)
and \EM{all} bicoloured paths (indicated by thick grey lines). The
right picture presents two tableaux of shapes $\pas{5, 3, 3, 2, 2, 1, 1, 1}/\pas{2, 1, 1, 1, 1}$
and $\pas{9, 5, 5, 1, 1, 1}/\pas{4, 3, 1}$, respectively, in the lower part: These tableaux
correspond to the overlay of lattice paths in the upper part, which is obtained by
\EM{recolouring} the bicoloured paths starting in $\pas{7,8}$ and in $\pas{-1,1}$ (these
bicoloured paths are indicated by thick grey lines, again.}
\label{fig:bicolouredPaths}
\input graphics/bicolouredPaths
\end{figure}

This construction is described in detail in \cite{fulmek:2001}. Here, we simply refer to
the left picture of \figref{fig:bicolouredPaths}, where the white paths are indicated by dashed
lines, and all bicoloured paths are indicated by thick grey lines.

The following observations are immediate:
\begin{obs}[{\bf Bicoloured paths always exist}]
\label{obs:arbitrary-point}
For \EM{every} coloured point $q$, there exists a bicoloured path starting at $q$.
\end{obs}

\begin{obs}[{\bf Bicoloured paths connect points of different radial orientation}]
\label{obs:different-orientation}
The bicoloured paths thus constructed never connect points of the same radial orientation
(i.e., two points oriented both towards or both away from the center).
\end{obs}

In the lower picture of \figref{fig:twoSkewShapes}, a possible pattern of ``connections
by bicoloured paths'' is indicated by dashed lines.

The following observation is easy to see:
\begin{obs}[{\bf Bicoloured path connect points of different parity}]
\label{obs:different-parity}
Two different bicoloured paths may have lattice points in common (they may intersect),
but they can never \EM{cross}.
If we assume some consecutive numbering of the coloured points in their circular
orientation (see the lower picture of \figref{fig:twoSkewShapes}), then this non--crossing condition implies that there can never be a bicoloured path connecting two points with numbers of the same parity.
\end{obs}

The non--crossing condition means that
if all such connections were drawn as straight lines connecting points on the circle, then
no two such lines can intersect. (In \figref{fig:twoSkewShapes}, not all connections are
drawn as straight lines for graphical reasons.) 


Consider some bicoloured path $b$ in the overlay of nonintersecting lattice paths $\pas{P_1,P_2}$: 
Changing colours (black to white and vice versa)
\bit
\item of both ending points of $b$
\item and of all arcs of $b$
\eit
gives a new overlay of nonintersecting lattice paths  $\pas{P_1^\prime,P_2^\prime}$
(with different starting/ending points). It is easy to see that we have for this
\EM{recolouring} of a bicoloured path:
\begin{obs}[{\bf Recolouring bicoloured paths is a weight preserving involution}]
\label{obs:weight-preserving-involution}
The recolouring of a bicoloured path $b$ in an overlay  of nonintersecting lattice
paths $\pas{P_1,P_2}$ is an \EM{involutive operation} (i.e., if we obtain
the overlay $\pas{P_1^\prime,P_2^\prime}$ by recolouring $b$ in $\pas{P_1,P_2}$, then
recolouring $b$ \EM{again} in $\pas{P_1^\prime,P_2^\prime}$ yields the original
$\pas{P_1,P_2}$), which \EM{preserves 
the respective weights}, i.e.,
$$
\weight\of{P_1}\cdot\weight\of{P_2}=\weight\of{P_1^\prime}\cdot\weight\of{P_2^\prime}.
$$
\end{obs}

\begin{figure}
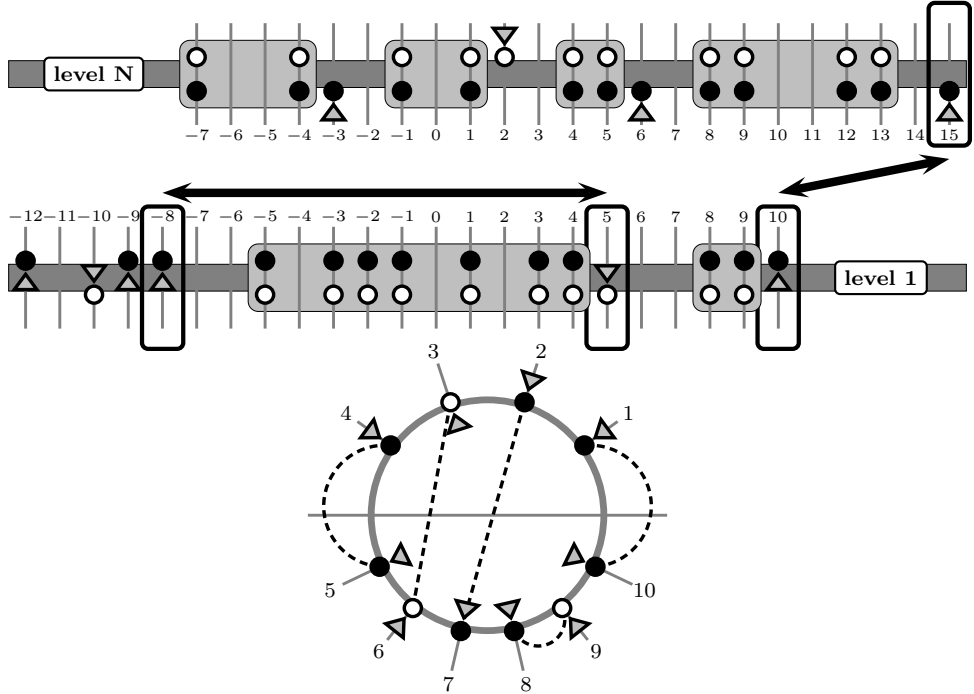

\caption{The two (necessarily different, by
Observation~\ref{obs:different-orientation}!)
bicoloured paths in \figref{fig:twoSkewShapes} which start at the white ending point
$\pas{15,N}$ and at the the black starting point $\pas{5,1}$ may have their respective other
ending points in $\pas{10, 1}$ and $\pas{-8,1}$. In the picture below, the corresponding points are marked by
white rectangles, the bicoloured paths are indicated by arrows. The picture shows the situation \EM{after} recolouring these paths.
}
\label{fig:twoSkewShapesRecoloured}
\input graphics/twoSkewShapesRecoloured
\end{figure}

Return to the example illustrated in \figref{fig:twoSkewShapes} and
consider the white ending point $\pas{15,N}$
and the black starting point $\pas{5,1}$ there. Note that both of these points are marked
with a triangle pointing \EM{downward}. The bicoloured paths ending at these points must have
their other ending points marked with a triangle pointing \EM{upward}. One possible choice of these other
ending points is depicted in \figref{fig:twoSkewShapesRecoloured}: The corresponding ending
points are marked by
white rectangles, the bicoloured paths are indicated by arrows. The picture shows the situation \EM{after} recolouring these paths.

The skew shape corresponding to the white points in \figref{fig:twoSkewShapesRecoloured} is
$\lambda^\prime/\mu^\prime$, where
\begin{align}
\lambda^\prime &=\pas{13, 13, 11, 11, 9, 9, 8, 8, 7, 5, 3},\notag\\
\mu^\prime&= \pas{9, 9, 7, 7, 7, 6, 5, 5, 5, 4, 0}.\label{eq:lambdaprime-muprime}
\end{align}
The skew shape corresponding to the black points in \figref{fig:twoSkewShapesRecoloured}
is  $\sigma^\prime/\tau^\prime$, where
\begin{align}
\sigma^\prime &= \pas{15, 14, 14, 12, 12, 11, 11, 11, 9, 8, 7, 7, 5},\notag\\
\tau^\prime &= \pas{10, 10, 10, 7, 7, 6, 5, 5, 5, 4, 2, 2, 0}.\label{eq:sigmaprime-tauprime}
\end{align}
For both skew shapes, the starting and ending points are shifted by $1$,
so, for instance, the starting point of the first path corresponding to
$\sigma^\prime/\tau^\prime$ is $\pas{\tau^\prime_1-1+1,1}=\pas{10,1}$ and the
ending point of the last (eleventh) path corresponding to $\lambda^\prime/\mu^\prime$
is $\pas{\lambda_{11}-11+1,N}=\pas{-7,N}$.

The pictures in \figref{fig:twoSkewShapesRecoloured} contain 
redundant information: All uncoloured points are ``doubled'', and the colour (black or white)
of the coloured points is determined uniquely
by their circular orientation. So we may encode the information in a more terse
way, namely as
\bit
\item the \EM{\pointconf} (see the upper picture in \figref{fig:GenericConfiguration}),
\item and the \EM{\pointarr} (see the lower picture in \figref{fig:GenericConfiguration}).
\eit
\begin{figure}
\caption{The information of \figref{fig:twoSkewShapesRecoloured} can be encoded in a more
terse way, namely as the \pointconf\ (shown in the upper picture) and the \pointarr\ 
(shown in the lower picture): The grey points in the upper picture are ``doubled''
points (coloured black \EM{and} white), and the colour of the white points in the upper
picture is determined by the \EM{orientation} of the corresponding points in the lower
picture (points in the upper half which are inwardly orientated and points in the lower
half which are outwardly oriented should be coloured black; all other points
should be coloured white).
}
\label{fig:GenericConfiguration}
\input graphics/GenericConfiguration
\end{figure}

We call a \pointarr\ \EM{admissible} if it has the same number of inwardly/outwardly
oriented points. \EM{Every} admissible \pointarr\ determines (together with the
corresponding \pointconf) a certain configuration of starting/ending points.
However, there might be no overlay of families of nonintersecting lattice paths that connect
these points (if, for instance, the $i$--th white ending point lies to the left of the
$i$--th white starting point; this would correspond to an $i$--th row of length $<0$
in the corresponding shape): In this case, the corresponding skew Schur function is
zero. But if there \EM{is} an overlay of families of nonintersecting lattice paths that connect
these points, then the family of all bicoloured paths determines a \EM{perfect matching} $M$
in the \pointarr\ (according to Observation~\ref{obs:arbitrary-point}), which is \EM{non--crossing}
(according to Observation~\ref{obs:different-parity}), and where all edges of $M$
connect points of different radial orientation (according to
Observation~\ref{obs:different-orientation}): We call such matchings \EM{admissible}. Note that recolouring some bicoloured
path amounts to reversing the orientation of the corresponding edge in the matching $M$.

We may summmarize all these considerations as follows (this is a reformulation of
\cite[Lemma~15]{fulmek:2001}):

\begin{lem}
\label{lem:fulmek1}
Let $\lambda/\mu$ and $\sigma/\tau$ be two skew shapes, and let $t$ be an arbitrary
integer. Consider the \pointconf\ corresponding to the starting/ending points with shift
$0$ for $\lambda/\mu$ and shift $t$ for $\sigma/\tau$. In the corresponding \pointarr,
choose a nonempty subset $S$ (arbitrary, but fixed) of the points oriented towards
the center.

Consider the set $V$ of \EM{all} admissible \pointarr s,
and consider the graph $G$ with vertex set $V$, where two vertices
$v_1$, $v_2$ are connected by an edge if and only if there are overlays of lattice
paths $\pas{P_1,P_2}$ and $\pas{P^\prime_1,P^\prime_2}$ for the starting/ending points corresponding to $\pas{\text{\pointconf},v_1}$
and $\pas{\text{\pointconf},v_2}$, respectively, such that $\pas{P^\prime_1,P^\prime_2}$
is obtained from $\pas{P_1,P_2}$ by recolouring \EM{all} bicoloured paths that are incident with
some point of $S$.

Obviously, this graph $G$ is \EM{bipartite} (i.e., $V=E\cup O$ with $E\cap O=\emptyset$,
such that there is no edge connecting two vertices of $E$ or two vertices of $O$).

Let $C$ be an arbitrary
connected component of $G$ with at least $2$ vertices, and denote by $C_O$
the set of pairs of skew shapes corresponding to $\pas{\text{\pointconf},x}$ for
$x\in O$, and by $C_E$
the set of pairs of skew shapes corresponding to $\pas{\text{\pointconf},x}$ for
$x\in E$.
Then we have the following identity for skew Schur functions:
\begin{equation}
\label{eq:lemma-general}
\sum_{\pas{\lambda/\mu,\;\sigma/\tau}\in C_E}
	\schurf_{\lambda/\mu}\cdot\schurf_{\sigma/\tau}
=
\sum_{\pas{\lambda^\prime/\mu^\prime,\;\sigma^\prime/\tau^\prime}\in C_O}
	\schurf_{\lambda^\prime/\mu^\prime}\cdot\schurf_{\sigma^\prime/\tau^\prime}.
\end{equation}
\end{lem}

This Lemma is rather unwieldy. But there is a particularly simple situation
which appears to be useful, so we state it as a Theorem (this is a reformulation of
\cite[Lemma~16]{fulmek:2001}):
\begin{thm}
\label{lem:fulmek2}
Under the assumptions of Lemma~\ref{lem:fulmek1}, let $c$ be the \pointarr\ for the
pair of shapes $\pas{\lambda/\mu,\;\sigma/\tau}$, and assume that the orientation of
the points in $c$  is \EM{alternating}. As in Lemma~\ref{lem:fulmek1}, let $S$
be some fixed subset of the points oriented towards the center in $c$.

Consider the set of all \pointarr s which can be obtained by reorienting all edges incident
with points in $S$ in some admissible matching of $c$, and denote the set of pairs of
skew shapes corresponding to such \pointarr s by $Q$.

Then we have:
\begin{equation}
\label{eq:lemma-special}
\schurf_{\lambda/\mu}\cdot\schurf_{\sigma/\tau}
=
\sum_{\pas{\lambda^\prime/\mu^\prime,\;\sigma^\prime/\tau^\prime}\in Q}
	\schurf_{\lambda^\prime/\mu^\prime}\cdot\schurf_{\sigma^\prime/\tau^\prime}.
\end{equation}
\end{thm}
\begin{proof}
Observe that in the right hand side of \eqref{eq:lemma-special} the Schur function product
$\schurf_{\lambda^\prime/\mu^\prime}\cdot\schurf_{\sigma^\prime/\tau^\prime}$
is either zero (if there is, in fact, no overlay of families of nonintersecting
lattice paths corresponding to the respective \pointarr), or there is some
corresponding overlay of families of nonintersecting
lattice paths $\pas{P_1^\prime,P_2^\prime}$. In the latter case, there \EM{are} bicoloured
paths starting in the points of $S$ (by Observation~\ref{obs:arbitrary-point}),
and by the combination of Observations \ref{obs:different-orientation} and \ref{obs:different-parity}, recolouring all such paths necessarily yields an overlay $\pas{P_1,P_2}$ of nonintersecting
lattice paths which corresponds to the pair $\pas{\lambda/\mu,\;\sigma/\tau}$.
\end{proof}
\section{Applications}
\label{sec:applications}
Clearly, the interpretation of Schur functions as generating functions of
$r$--tuples of nonintersecting lattice paths is best suited for the bijective
construction of recolouring bicoloured paths. But of course, the recolouring
operation can be translated into operations for the shapes of the corresponding tableau
(i.e., for the corresponding partitions, or equivalently, Ferrers diagrams).

We shall show how this translation gives the identity \cite[(3.3)]{gurevichPyatovSaponov:2009}
of Gurevich, Pyatov and Saponov, but before doing this we consider a simple special case,
in order to illustrate the meaning of Theorem~\ref{lem:fulmek2}:
\begin{ex}
\label{ex:fulmek}
Assume that for the shapes $\lambda/\mu$ and $\sigma/\tau$  we have
\bit
\item $\mu=\tau$,
\item $\lambda_1>\sigma_1$,
\item $\length{\lambda}=\length{\sigma}$.
\eit
Choose shift $0$ for the families of starting/ending points corresponding to these
shapes, then there is no coloured starting point (since $\mu=\tau$): accordingly, only the
ending points are shown in \figref{fig:twoSimpleShapes}. Furthermore,
assume that the corresponding black and white ending points \EM{alternate} along level $N$: Then
the preconditions of Theorem~\ref{lem:fulmek2} are fulfilled.
Since $\lambda_1>\mu_1$, the point $q=\pas{\lambda_1-1,N}$ is white. Consider
the set $S=\setof{q}$: The
bicoloured path $b$ starting in $q$ necessarily must end in a black point
(by Observation~\ref{obs:different-orientation}). Assume that there are $k$ such
black points $q_1,\dots, q_n$, and let $\indexthis{\lambda}{i}$ and $\indexthis{\sigma}{i}$ be the
partitions corresponding to the configuration of white and black points
obtained by changing colours of $q$ and $q_i$, $i=1,\dots,k$ (i.e., colour $q$ black
and $q_i$ white, and leave all other colours unchanged). Then by Theorem~\ref{lem:fulmek2}
we have:
$$
\schurf_{\lambda/\mu}\cdot\schurf_{\sigma/\mu} =
\sum_{i=1}^k \schurf_{\indexthis{\lambda}{i}/\mu}\cdot\schurf_{\indexthis{\sigma}{i}/\mu}
$$
\end{ex}
\begin{figure}
\caption{Application of Lemma~\ref{lem:fulmek2} to the special case $\mu=\tau$, $\lambda_1>\sigma_1$,
$\length{\lambda}=\length{\sigma}$, and shift $0$ for \EM{all} starting/ending points:
Since there is no coloured starting point in this case, only the
ending points (on level $N$) are shown.
The point $\pas{15,N}$ (ending point of bicoloured path $b$) is marked with a white
rectangle in the upper row. The middle row and the lower row show the configurations
which can arise by recolouring $b$; the respective other ending points $\pas{6,N}$
and $\pas{-3,N}$ are again marked with a white
rectangle.
}
\label{fig:twoSimpleShapes}
\input graphics/twoSimpleShapes
\end{figure}

\figref{fig:twoSimpleShapes} illustrates this example for
\begin{align*}
\lambda&=\pas{16, 15, 15, 13, 13, 11, 11, 10, 10, 9, 7, 5},\\
\sigma&= \pas{14, 14, 12, 12, 11, 11, 11, 9, 8, 7, 7, 5}.
\end{align*}
From the pictures in \figref{fig:twoSimpleShapes} we see that $k=2$ in this case, with
\begin{align*}
\indexthis{\lambda}{1}&=\pas{14, 14, 12, 12, 11, 11, 11, 10, 10, 9, 7, 5},\\
\indexthis{\sigma}{1}&= \pas{16, 15, 15, 13, 13, 11, 11, 9, 8, 7, 7, 5},
\end{align*}
(shown in the middle row of \figref{fig:twoSimpleShapes}) and 
\begin{align*}
\indexthis{\lambda}{2}&=\pas{14, 14, 12, 12, 10, 10, 9, 9, 8, 7, 7, 5},\\
\indexthis{\sigma}{2}&= \pas{16, 15, 15, 13, 13, 12, 12, 12, 10, 9, 7, 5}
\end{align*}
(shown in the lower row of \figref{fig:twoSimpleShapes}).
\figref{fig:FerrersDiagrams} presents the Ferrers diagrams for this example,
where we chose $\mu=\tau=\pas{5,4,2,2,2,1,1,1}$.

\begin{figure}
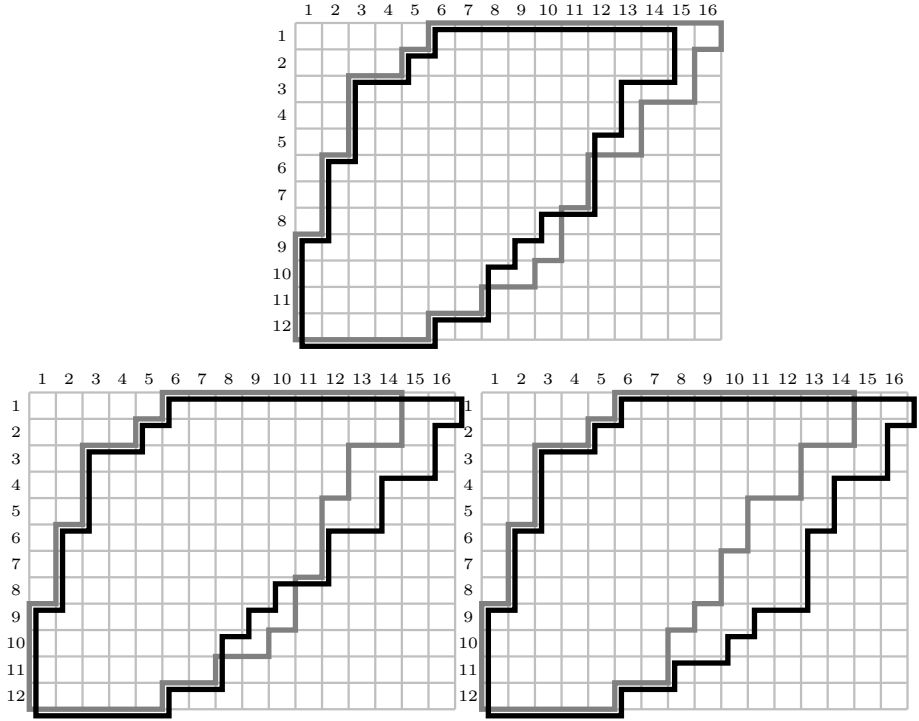

\caption{The skew Ferrers boards corresponding to the example in \figref{fig:twoSimpleShapes}
for $\mu=\tau=\pas{5,4,2,2,2,1,1,1}$: The upper picture shows the Ferrers board of
$\lambda/\mu$ (drawn with grey lines) and $\sigma/\tau$ (drawn with black lines). The lower
pictures show the Ferrers boards of $\indexthis{\lambda}{i}/\mu$ and $\indexthis{\sigma}{i}/\tau$, $i=1,2$.
}
\label{fig:FerrersDiagrams}
\input graphics/FerrersDiagrams
\end{figure}

\subsection{The identity of Gurevich, Pyatov and Saponov}
\label{sec:saponov}
Now consider the special case $\mu=\tau$, $\lambda_1>\sigma_1$,
$\length{\lambda}=\length{\sigma}+1$, with shift $0$ for all starting and ending points,
and where black and white points alternate in their circular orientation.
As in Example~\ref{ex:fulmek}, let $q=\pas{\lambda_1-1,N}$ and choose $S=\setof{q}$.
The possible ending points of the bicoloured path $b$ ending in $q$
are
\bit
\item the black points at level
$N$
\item and the leftmost white starting point $q_0$.
\eit
As our running example we choose the skew shapes $\lambda/\mu$ and $\sigma/\tau$ with
\begin{align*}
\lambda &= \pas{10,7,7,6,6,4,4,3,2,2}, \\
\sigma &= \pas{8,7,7,5,4,4,2,1,1}, \\
\mu=\tau &= \pas{4,3,3,1}.
\end{align*}

\begin{figure}
\caption{Application of Theorem~\ref{lem:fulmek2} to the special case $\mu=\tau$, $\lambda_1>\sigma_1$,
$\length{\lambda}=\length{\sigma}+1$ and shift $0$ for all starting and ending points: Consider the configurations that
can arise by recolouring the bicoloured path $b$ ending in $q=\pas{\lambda_1-1,N}$. The uppermost
picture shows the configuration of the starting and ending points of $\lambda/\mu$ (drawn as
white circles) and
$\sigma/\tau$ (drawn as black circles), where $\mu=\tau=\pas{4,3,3,1}$. The point $q=\pas{\lambda_1-1,N}$ is marked by a
white rectangle.
\newline
The three pictures below show the three possible configurations arising by the recolouring
of $b$; the other ending point of $b$ is marked
by a white rectangle.
\newline
From Theorem~\ref{lem:fulmek2} we obtain the Schur function identity
$$
\schurf_{\lambda/\mu}\cdot\schurf_{\sigma/\tau}=
\schurf_{\indexthis{\lambda}{0}/\mu}\cdot\schurf_{\indexthis{\sigma}{0}/\mu}+
\schurf_{\indexthis{\lambda}{1}/\mu}\cdot\schurf_{\indexthis{\sigma}{1}/\mu}+
\schurf_{\indexthis{\lambda}{2}/\mu}\cdot\schurf_{\indexthis{\sigma}{2}/\mu}.
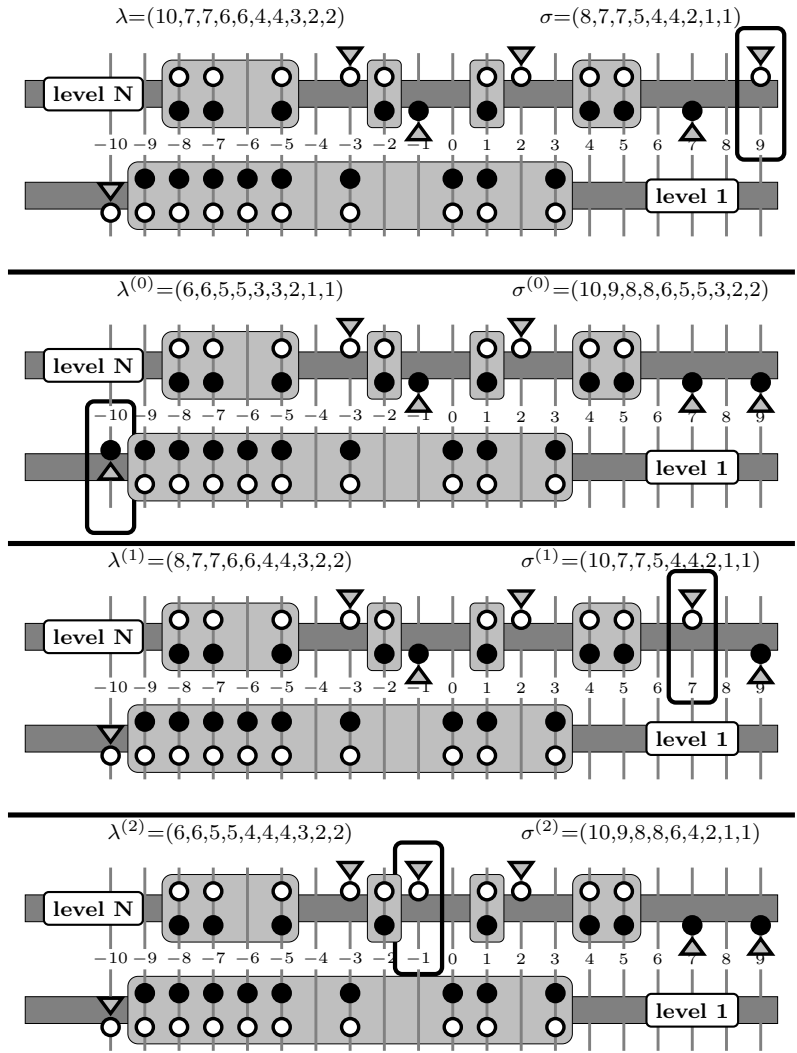
$$
}
\label{fig:SaponovLatticePaths}
\input graphics/SaponovLatticePaths
\end{figure}

See the upper picture in \figref{fig:SaponovLatticePaths}
for an illustration.

Assume that there are $k>1$ black points $q_1,\dots,q_k$ and denote the
shapes corresponding to recolouring $q$ and $q_i$, $i=0,1,\dots,k$, by
$\indexthis{\lambda}{i}/\mu$ and $\indexthis{\sigma}{i}/\mu$, respectively.
Then by Theorem~\ref{lem:fulmek2}
we have:
\begin{equation}
\label{eq:saponov-fulmek}
\schurf_{\lambda/\mu}\cdot\schurf_{\sigma/\mu} =
\schurf_{\indexthis{\lambda}{0}/\mu}\cdot\schurf_{\indexthis{\sigma}{0}/\mu} +
\sum_{i=1}^k \schurf_{\indexthis{\lambda}{i}/\mu}\cdot\schurf_{\indexthis{\sigma}{i}/\mu}
\end{equation}

See the three lower pictures in \figref{fig:SaponovLatticePaths}
for an illustration (in this example, $k=2$).

If we choose $\mu=0$, \eqref{eq:saponov-fulmek} amounts precisely to the identity \cite[(3.3)]{gurevichPyatovSaponov:2009}: We simply
have to \EM{translate} our formulation to the language of adding and removing \EM{partial border strips}
to Ferrers diagrams, which was used by Gurevich, Pyatov and Saponov
\cite{gurevichPyatovSaponov:2009}. To get a first idea, have a look at
\figref{fig:Saponov}, which presents the Ferrers diagrams corresponding to the concrete
example of \figref{fig:SaponovLatticePaths}.

\begin{figure}
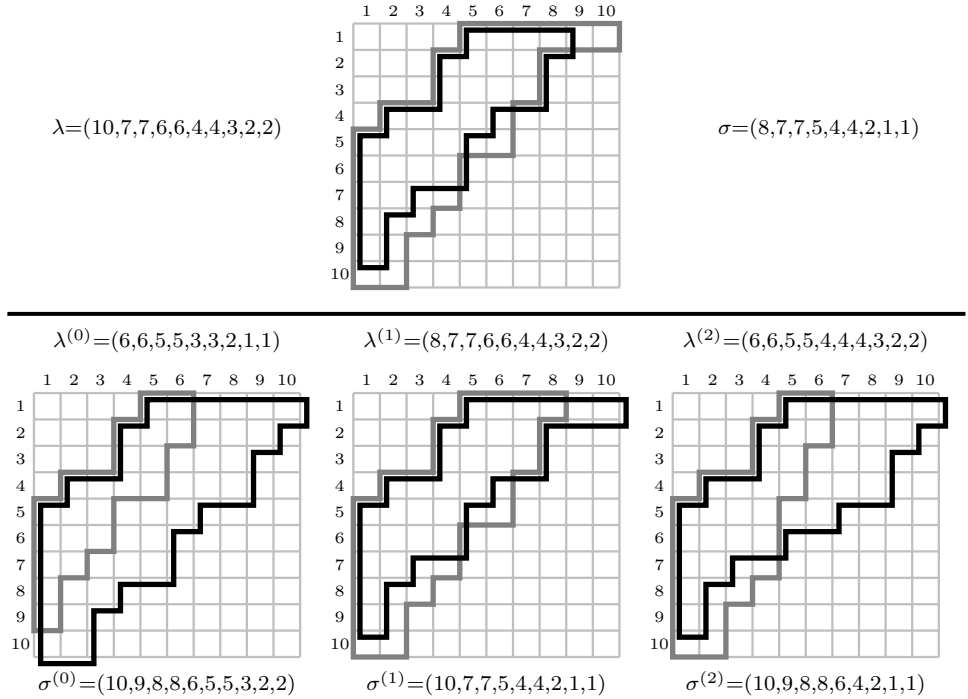

\caption{The skew Ferrers diagrams corresponding to the example in \figref{fig:SaponovLatticePaths}: The upper picture shows the overlay of the
Ferrers diagrams for $\lambda/\mu$ and $\sigma/\mu$ (without recolouring), and the
lower pictures show the overlay of diagrams $\indexthis{\lambda}{i}/\mu$ and
$\indexthis{\sigma}{i}/\mu$,
$i=0,1,2$.
}
\label{fig:Saponov}
\input graphics/Saponov
\end{figure}

Observe that
recolouring starting/ending points may be viewed as a game of inserting/removing points in
the configuration of starting/ending points corresponding to some partition. Instead of
giving a lengthy verbal description, we present in \figref{fig:Recolouring} the
effect of inserting (read the upper picture upwards, from $\sigma$ to $\lambda$) or removing
(read the upper picture downwards, from $\lambda$ to $\sigma$) points in a graphical way.

Reading the upper
part of \figref{fig:Recolouring} \EM{downwards} (i.e., removing the point in position $\lambda_i-i$),
the Ferrers diagram of $\sigma$ is obtained
from the Ferrers diagram of $\lambda$ by a \EM{down--peeling} of a \EM{partial border strip} starting at
row $i$, or, in the language of \cite{gurevichPyatovSaponov:2009}:
$$\sigma=\sappeeldown{\lambda}{i}.$$

The special case of this operation for $i=1$ amounts to the removal of the \EM{complete}
border strip --- in the language of \cite{gurevichPyatovSaponov:2009}:
$$
\sappeel{\lambda}\defeq\sappeeldown{\lambda}{1}.
$$

%

\begin{figure}
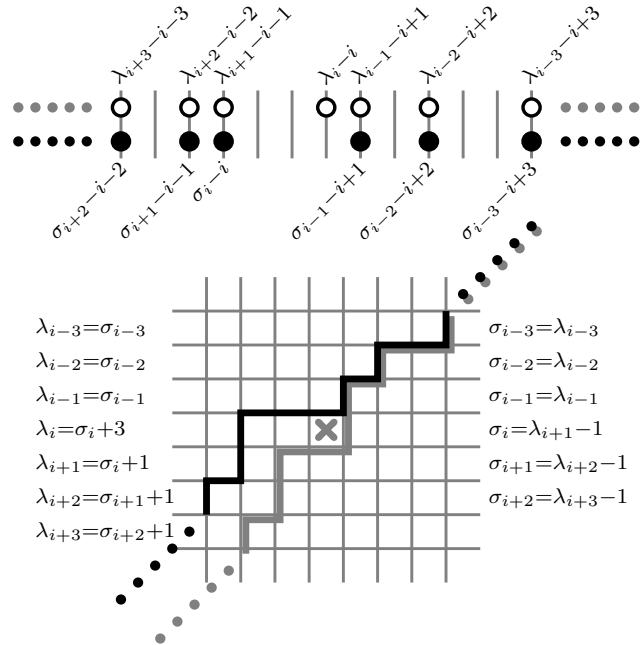

\caption{
The operation of inserting/removing points in the configuration of
starting/ending points can be translated to the language
introduced in \cite{gurevichPyatovSaponov:2009}, i.e., to operations on Ferrers diagrams:
In the lower part of the picture, the grey lines correspond to the Ferrers diagram of $\lambda$, and the
black lines correspond to the Ferrers diagram of $\sigma$. Reading the upper
part of the picture \EM{downwards} (i.e., removing the point in position $\lambda_i-i$),
the Ferrers diagram of $\sigma$ is obtained
from the Ferrers diagram of $\lambda$ by a \EM{down--peeling} of a \EM{partial border strip} starting at
row $i$: $\sigma=\sappeeldown{\lambda}{i}$. The partial border strip is the area between the
black and grey shapes in rows $i,i+1,\dots$: It is clear that such strip begins with the
last box in row $i$ (indicated by a grey "x" in the picture).
}
\label{fig:Recolouring}
\input graphics/Recolouring
\end{figure}

Now observe that a pair of partitions $\pas{\lambda,\sigma}$ fulfilling the
preconditions
for \eqref{eq:saponov-fulmek} can be obtained by \EM{constructing} an appropriate $\sigma$
for a given $\lambda$. This is achieved by applying to the configuration of ending points
corresponding to $\lambda$ a sequence of $k$ removals/insertions of points, followed by
removing
the right--most ending point and inserting a new left--most starting point, see
the upper picture in \figref{fig:SaponovLatticePaths}.

Again, we shall illustrate the simple procedure by pictures
instead of giving a lengthy verbal description.

The first step of this construction is illustrated in \figref{fig:SaponovTranslation}:
The configuration of points corresponding to the partition
$\lambda=\pas{10,7,7,6,6,4,4,3,2,2}$ considered in
\figref{fig:SaponovLatticePaths} is presented in the upper part of the picture.
This configuration is changed by adding a new point at position $7$ first (the
result of this change is presented in the middle part of the picture) and then by
removing the point at position $2$ (the
result of this change is presented in the lower part of the picture). It is obvious
that this amounts to adding a \EM{partial border strip} to the Ferrers diagram
of $\lambda$, which consists of $t_1=2$ boxes in row $r_1=2$ and spans the $m_1=3$ rows
$2,3$ and $4$. Stated in the language introduced in \cite{gurevichPyatovSaponov:2009}, the
picture in the lower row of \figref{fig:SaponovTranslation} shows
$$\sapadd{\lambda}{t_1}{(r_1,m_1)}.$$ 

\begin{figure}
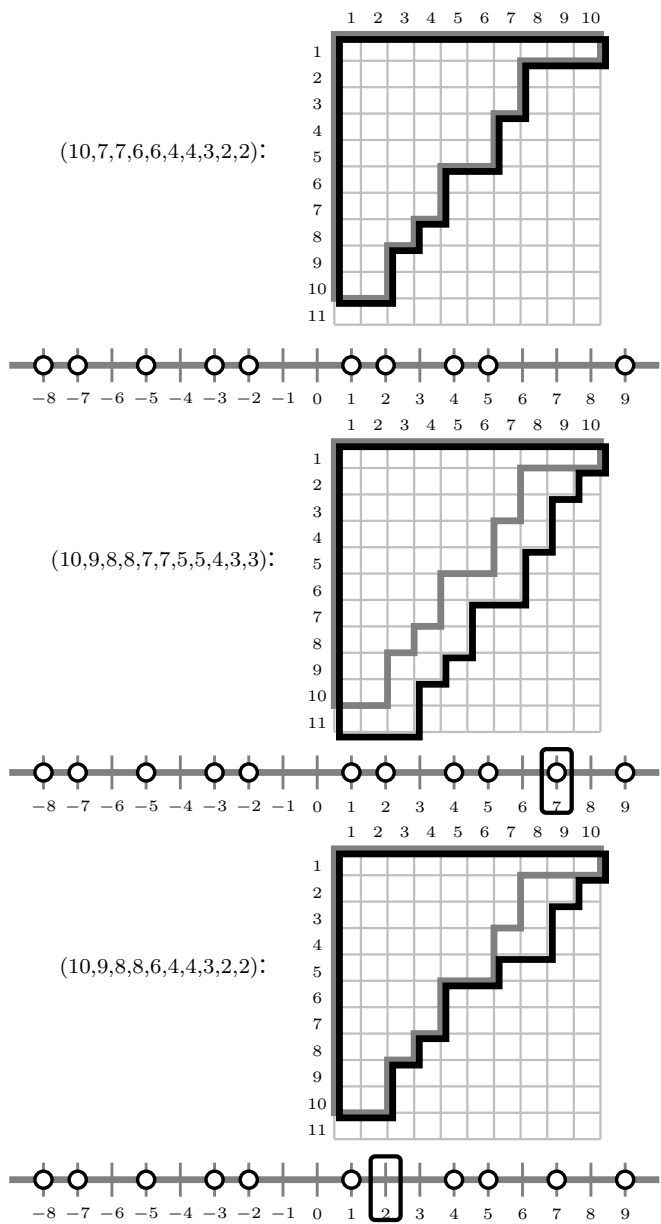

\caption{Adding some point at position $i$ and removing a point at position $j<i$
amounts to adding a \EM{partial border strip}. The pictures show the Ferrers
diagram of the original partition $\lambda$ with grey lines, and the Ferrers
diagrams of the partitions obtained by adding/removing points with black lines.
The corresponding configuration of ending points is depicted unter the Ferrers
diagrams.
}
\label{fig:SaponovTranslation}
\input graphics/SaponovTranslation
\end{figure}

The next step of this construction is to add a point at position $-1$ and to remove the point
at position $-3$. This amounts to adding a partial border strip consisting of $t_2=2$
boxes in row $r_2=6$ and spanning the $m_2=2$ rows $6$ and $7$.
We thus obtain $\nu=\pas{10,9,8,8,6,5,5,3,2,2}$, or, in the language of \cite{gurevichPyatovSaponov:2009}: $\nu=\sapadd{\lambda}{2,1}{(2,3),(6,2)}$,
see the upper picture in \figref{fig:SaponovTranslation2}.

Finally, we remove the rightmost ending point at position $9$: This amounts to removing
a \EM{complete} border strip from the Ferrers diagram of $\nu$, giving the partition
$\sigma=\sappeel{\nu}=\pas{8,7,7,5,4,4,2,1,1}$ of our running example, see the lower picture in \figref{fig:SaponovTranslation2}.

\begin{figure}
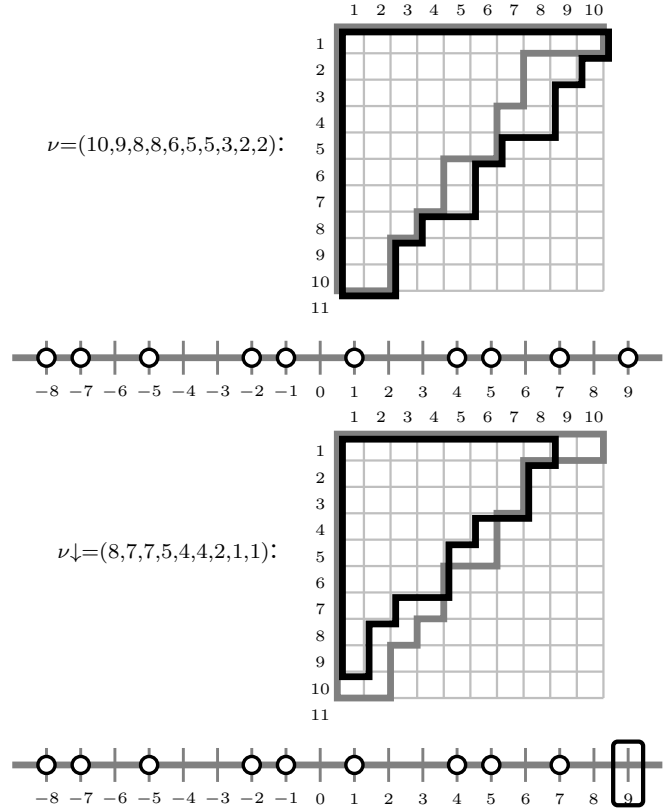

\caption{Removing the point corresponding to $\nu_1$ (marked by a white rectangle in
the lower picture) amounts to removing
the \EM{complete border strip} spanning rows $1, 2,\dots,\length{\nu}$ from
the Ferrers board of $\nu$,
see \figref{fig:Recolouring}.
}
\label{fig:SaponovTranslation2}
\input graphics/SaponovTranslation2
\end{figure}

So we see that the Schur function identity stated in \figref{fig:SaponovLatticePaths}
can be partially translated as:
$$
\schurf_{\lambda/\mu}\cdot\schurf_{\sappeel{\nu}/\tau}=
\schurf_{\sappeel{\lambda}/\mu}\cdot\schurf_{\nu/\mu}+\cdots
$$

\begin{figure}
\caption{Removing the point corresponding to $\lambda_1$ and inserting
a new point at positions $7$ and $-1$, respectively (these points are marked by
white rectangles in the pictures) yield the partitions $\indexthis{\lambda}{1}$ and
$\indexthis{\lambda}{2}$ from \figref{fig:SaponovLatticePaths}, respectively.
In the language of \cite{gurevichPyatovSaponov:2009}, $\indexthis{\lambda}{1}$
is obtained by the \EM{up--peeling} of a partial border strip starting in row $1$ at
the second possible box (marked by a grey ``x'' in the upper picture), and
$\indexthis{\lambda}{2}$
is obtained by the up--peeling of a partial border strip starting in row $5$ at
the first possible box (marked by a grey ``x'' in the upper picture):
$\indexthis{\lambda}{1}=\sappeelup{\lambda}{1,2}$ and
$\indexthis{\lambda}{2}=\sappeelup{\lambda}{5,1}$.
}
\label{fig:SaponovTranslation3}
\input graphics/SaponovTranslation3
\end{figure}

To complete this translation, have a look at \figref{fig:SaponovTranslation3}
and note that the partitions $\indexthis{\lambda}{1}$ and $\indexthis{\lambda}{2}$
(drawn with black lines) are obtained by the original
partition $\lambda$ (drawn with grey lines) by removing a partial border strip
which starts in the box marked with a small ``x'' and extends up to the first row.
Note that for such \EM{up--peeling} of a border strip starting at row $i$, there are
$\lambda_i-\lambda_{i+1}$ possible positions of the starting box: Number them from
left to right, then the up--peeling is uniquely determined by the row number
$i$ and the box number $t$. In 
the language of \cite{gurevichPyatovSaponov:2009}, this is denoted by
$$
\sappeelup{\lambda}{i,t},
$$
i.e., we have
$
\indexthis{\lambda}{1}=\sappeelup{\lambda}{1,2}
$
and
$
\indexthis{\lambda}{2}=\sappeelup{\lambda}{5,1}.
$

Putting all these
observations together, we see that the Schur function identity stated in \figref{fig:SaponovLatticePaths}
can be translated as:
$$
\schurf_{\lambda/\mu}\cdot\schurf_{\sappeel{\nu}/\tau}=
\schurf_{\sappeel{\lambda}/\mu}\cdot\schurf_{\nu/\mu}+
\schurf_{\sappeelup{\lambda}{1,2}/\mu}\cdot\schurf_{\sappeeldown{\nu}{2}/\mu}+
\schurf_{\sappeelup{\lambda}{5,1}/\mu}\cdot\schurf_{\sappeeldown{\nu}{6}/\mu}.
$$

So it is clear that the special case considered in this
section can be stated as follows:

\begin{cor}[Gurevich, Pyatov and Saponov]
Let $\lambda=\pas{\lambda_1,\dots,\lambda_r}$ be a partition. Assume that there
are $k$ indices $2\leq r_1<\dots<r_k\leq r$ such that $\lambda_{r_i}<\lambda_{r_{i-1}}$,
$i=1,\dots, r$.
Choose integers $t_i$ and $m_i$ for $i=1,\dots,k$ subject to the restrictions

\begin{align*}
1 &\leq t_i\leq\lambda_{r_i-1}-\lambda_{r_i}, \\
1 &\leq m_i\leq r_{i+1}-r_i.
\end{align*}

Then we may construct $\nu=\sapadd{\lambda}{t_1,\dots,t_k}{(r_1,m_1)\dots(r_k,m_k)}$, and
we have
$$
\schurf_{\lambda}\cdot\schurf_{\sappeel{\nu}}=
\schurf_{\sappeel{\lambda}}\cdot\schurf_{\nu}+
\sum_{i=1}^k\schurf_{\sappeelup{\lambda}{r_i-1,t_i}}\cdot
	\schurf_{\sappeeldown{\nu}{(r_i)}}.
$$

\end{cor}

\bibliography{paper}

\end{document}